\documentclass{article}
\usepackage{amsmath, amssymb, amsthm}
\usepackage[top=25truemm, bottom=25truemm, left=25truemm, right=25truemm]{geometry}
\usepackage{accents}
\usepackage{algorithm,algorithmic}
\usepackage{ascmac}
\usepackage{cases}
\usepackage{tikz-cd}
\usepackage{titlesec}
\makeatletter
\def\sbar{\accentset{{\cc@style\underline{\mskip13mu}}}}
\def\mbar{\accentset{{\cc@style\underline{\mskip25mu}}}}
\makeatother

\theoremstyle{definition}
\newtheorem{Def}{Definition}[section]
\newtheorem{Lem}[Def]{Lemma}
\newtheorem{Prop}[Def]{Proposition}
\newtheorem{Thm}[Def]{Theorem}
\newtheorem{Cor}[Def]{Corollary}

\newtheorem{Rmk}[Def]{Remark}
\theoremstyle{plain}
\newtheorem{Main}{Main Theorem}

\title{On the Rosenhain forms of superspecial curves of genus two}
\author{Ryo Ohashi}

\begin{document}
\maketitle
\begin{abstract}
In this paper, we examine superspecial genus-2 curves $C: y^2 = x(x-1)(x-\lambda)(x-\mu)(x-\nu)$ in odd characteristic $p$.
As a main result, we show that the difference between any two elements in $\{0,1,\lambda,\mu,\nu\}$ is a square in $\mathbb{F}_{p^2}$.
Moreover, we show that $C$ is maximal or minimal over $\mathbb{F}_{p^2}\!$ without taking its $\mathbb{F}_{p^2}$-form (we also give a criterion in terms of $p$ that tells whether $C$ is maximal or minimal).
As these applications, we study the maximality of superspecial hyperelliptic curves of genus $3$ and $4$ whose automorphism groups contain $\mathbb{Z}/2\mathbb{Z} \times \mathbb{Z}/2\mathbb{Z}$.
\end{abstract}

\section{Introduction}\label{Introduction}
Throughout this paper, a curve always means a non-singular projective variety of dimension one defined over a field of characteristic $p \geq 3$.
An elliptic curve $E$ is called {\it supersingular} if the $p$-torsion group of $E$ is trivial.
Recently, supersingular elliptic curves are often used in protocols of isogeny-based cryptosystems.
One of the reasons is that all supersingular elliptic curves are defined over $\mathbb{F}_{p^2}$, and so we can do computations without a field extension any further.
As a more advanced result by Auer-Top \cite{AT}, they investigated Legendre forms of supersingular elliptic curves; if an elliptic curve $E: y^2 = x(x-1)(x-t)$ is supersingular, then $-t$ is eighth power in $\mathbb{F}_{p^2}$.

In this paper, we focus mainly on superspecial genus-2 curves.
Here, a curve $C$ is called {\it superspecial} when the Jacobian variety of $C$ is isomorphic to the product of supersingular elliptic curves.
It is known \cite{Nygaard} that the curve $C$ is superspecial if and only if the Cartier operator on $H^0(C,\varOmega_C)$ vanishes.
Superspecial curves are not only important objects in algebraic geometry, but also have applications to such as cryptography and coding theory.
Here, let us review previous works on superspeciality of genus-2 curves briefly: Ibukiyama-Katsura-Oort determined the exact number of isomorphism classes of superspecial genus-2 curves in \cite[Theorem 3.3]{IKO}.
In particular, there is such a curve for arbitrary characteristics $p \geq 5$.
Jordan-Zaytman \cite[Section 7]{JZ} showed that the superspecial $(2,2)$-isogeny graph is connected, which implies that all superspecial genus-2 curves can be listed by using $(2,2)$-isogenies (see Subsection \ref{Richelot} for detail).
Katsura-Takashima counted the number of superspecial $(2,2)$-isogenies in \cite[Section 6]{KT}.
In terms of application, Castryck-Decru-Smith\,\cite{CDS} constructed hash functions using superspecial genus-2 curves.

Our first contribution on superspecial genus-2 curves is that we give a variant of Auer-Top's result.
More precisely, we obtain the following result on Rosenhain forms of superspecial genus-2 curves:
\begin{Main}
Assume that the genus-2 curve
\[
    C: y^2 = x(x-1)(x-\lambda)(x-\mu)(x-\nu)
\]
is superspecial. Then, the following statements are true:
\begin{enumerate}
\item All the 9 values
\[
    \lambda,\,\mu,\,\nu,\,1-\lambda,\,1-\mu,\,1-\nu,\,\lambda-\mu,\,\mu-\nu,\,\nu-\lambda
\]
are squares in $\mathbb{F}_{p^2}$.
\item All the 5 values
\[
    \lambda\mu\nu,\,(1-\lambda)(1-\mu)(1-\nu),\,\lambda(\lambda-1)(\lambda-\mu)(\lambda-\nu),\,\mu(\mu-1)(\mu-\lambda)(\mu-\nu),\,\nu(\nu-1)(\nu-\lambda)(\nu-\mu)
\]
are fourth powers in $\mathbb{F}_{p^2}$.
\end{enumerate}
\end{Main}

The other main result in this paper concerns the maximality of genus-2 curves. Here, a curve $C$ is called {\it maximal} (resp.\ {\it minimal}) over $\mathbb{F}_q$ with $q=p^{2e}$ if the number of $\mathbb{F}_q$-rational points of $C$ attains the Hasse-Witt upper (resp.\ lower) bound.
It is known that maximal or minimal curves over $\mathbb{F}_{p^2}\hspace{-0.3mm}$ are all superspecial, whereas all superspecial curves over $\mathbb{F}_{p^2}\hspace{-0.3mm}$ are not necessarily maximal or minimal (cf. \cite[Section II, Theorem 1.1]{Ekedahl}).
On the other hand, it follows from Auer-Top's result \cite[Proposition 2.2]{AT} that all supersingular Legendre elliptic curves $E: y^2 = x(x-1)(x-t)$ are maximal or minimal over $\mathbb{F}_{p^2}$.
In \cite[Theorem 1.1]{Ohashi1}, the author studied the maximality of hyperelliptic genus-3 curves $H: y^2 = (x^4-ax^2+1)(x^4-bx^2+1)$, whose automorphism groups contain $(\mathbb{Z}/2\mathbb{Z})^3$; if $H$ is superspecial, then $H$ is maximal or minimal over $\mathbb{F}_{p^2}$.
In \cite[Theorem 1.1]{Ohashi2}, the author also studied the maximality of Ciani quartics $H'\hspace{-0.3mm}: x^4 + y^4 + z^4 + rx^2y^2 + sy^2z^2 + tz^2x^2 = 0$, which are non-hyperelliptic genus-3 curves whose automorphism groups contain the dihedral group of order $4$; if $H'$ is superspecial, then $H'$ is maximal or minimal over $\mathbb{F}_{p^2}$.
In this paper, we give a variant of these results for superspecial genus-2 curves:
\begin{Main}
Assume that the genus-2 curve
\[
    C: y^2 = x(x-1)(x-\lambda)(x-\mu)(x-\nu)
\]
is superspecial. Then, the curve $C$ is maximal or minimal over $\mathbb{F}_{p^2}$. More precisely, we have the following:
\begin{itemize}
\item The case of $p \equiv 3 \pmod{4}$: the curve $C$ is maximal over $\mathbb{F}_{p^2}$.
\item The case of $p \equiv 1 \pmod{4}$: the curve $C$ is minimal over $\mathbb{F}_{p^2}$.
\end{itemize}
In particular, the curve $C$ is maximal or minimal over $\mathbb{F}_{p^2}$.
\end{Main}

As an application of Main Theorem B, we examine hyperelliptic genus-3 curves $D$ whose automorphism groups contain $(\mathbb{Z}/2\mathbb{Z})^2$ and hyperelliptic genus-4 curves $D'$ whose automorphism groups contain $(\mathbb{Z}/2\mathbb{Z})^2$ in Section \ref{Application}.
We find a form of $D$ (resp. $D'$) such that its superspeciality implies its maximality or minimality over $\mathbb{F}_{p^2}$, and we give an explicit criterion whether $D$ (resp. $D'$) is maximal or minimal over $\mathbb{F}_{p^2}$.\par\smallskip
Now, the rest of this paper is organized as follows: Section 2 is devoted to preliminaries for genus-2 curves.
In Subsections 3.1 and 3.2, we prove Main Theorems A and B respectively.
Finally, we give three applications of these main theorems in Section 4.\vspace{-1.5mm}

\subsection*{Acknowledgements}
The author thank Prof.\,Shushi Harashita for his helpful comments on the abstract description of $(2,2)$-isogeny between the Jacobian varieties of genus-2 curves.
This research was conducted under a contract of "Research and development on new generation cryptography for secure wireless communication services" among "Research and Development for Expansion of Radio\,Wave Resources (JPJ000254)", which was supported by the Ministry of Internal Affairs and Communications, Japan.\vspace{-1.5mm}

\section{Preliminaries}\label{Preliminaries}
In this section, we collect some results on genus-2 curves. In Subsection \ref{Rosenhain}, we give an explicit transformation from a given genus-2 curve into its Rosenhain form.
In Subsection \ref{RA}, we review the classification of genus-2 curves by reduced automorphism groups. In Subsection \ref{Richelot}, we recall how to enumerate superspecial genus-2 curves by using Richelot isogeny.
Let $K$ be a field of characteristic $p \geq 3$ throughout this section. 

\subsection{Rosenhain forms of genus-2 curve}\label{Rosenhain}
Any genus-2 curve has just 6 Weierstrass points, and we can consider an isomorphism $C \cong C_{\lambda,\mu,\nu}\!$ which maps three of them to $0,1$ and $\infty$.
\begin{Def}
Given a genus-2 curve $C$, we say that
\[
    C_{\lambda,\mu,\nu}: y^2 = x(x-1)(x-\lambda)(x-\mu)(x-\nu)
\]
is a {\it Rosenhain form} of $C$ when there exists an isomorphism $C \cong C_{\lambda,\mu,\nu}$ over the algebraic closure of $K$. Then the values $\lambda,\mu$ and $\nu$ are called {\it Rosenhain invariants} of $C_{\lambda,\mu,\nu}$.
\end{Def}
Consider the genus-2 curve
\begin{equation}\label{sextic}
    C: Y^2 = c(X-a_1)(X-a_2)(X-a_3)(X-a_4)(X-a_5)(X-a_6), \quad c \in K^\times,
\end{equation}
where each $a_i$ is a distinct element of $ \sbar{K} \cup \{\infty\}$.
If $a_i = \infty$, then we mean that the factor $(X-a_i)$ is excluded from the above equation.
The transformation
\[
    X \mapsto \frac{X-a_1}{X-a_3} \cdot \frac{a_2-a_3}{a_2-a_1} =: x, \ \, Y \mapsto \frac{c^{-1}Y}{(a_3-a_1)(a_3-a_2)(a_3-a_4)(a_3-a_5)(a_3-a_6)}\biggl(\frac{a_3-a_1}{X-a_3}\cdot\frac{a_2-a_3}{a_2-a_1}\bigg)^{\!\!3} =: y
\]
gives the equation defined by
\[
    \kappa y^2 = x(x-1)\biggl(x-\frac{(a_4-a_1)(a_2-a_3)}{(a_4-a_3)(a_2-a_1)}\biggr)\biggl(x-\frac{(a_5-a_1)(a_2-a_3)}{(a_5-a_3)(a_2-a_1)}\biggr)\biggl(x-\frac{(a_6-a_1)(a_2-a_3)}{(a_6-a_3)(a_2-a_1)}\biggr)
\]
with
\begin{equation}\label{epsilon}
    \kappa := c(a_1-a_2)(a_3-a_4)(a_3-a_5)(a_3-a_6).
\end{equation}
Note that $\kappa$ is a non-zero element of $\sbar{K}$ since $i \neq j \Rightarrow a_i \neq a_j$ by the assumption. Setting
\begin{equation}\label{lmn}
    \lambda := \frac{(a_4-a_1)(a_2-a_3)}{(a_4-a_3)(a_2-a_1)}, \quad \mu := \frac{(a_5-a_1)(a_2-a_3)}{(a_5-a_3)(a_2-a_1)}, \quad \nu := \frac{(a_6-a_1)(a_2-a_3)}{(a_6-a_3)(a_2-a_1)},
\end{equation}
then we see that $C_{\lambda,\mu,\nu}$ is a Rosenhain form of $C$.
\begin{Rmk}
Since the number of choices of $(a_1,a_2,a_3)$ is equal to $6 \times 5 \times 4 = 120$, there are 120 Rosenhain forms of $C$ at most (cf. \cite[Lemma 5]{OKH}).
\end{Rmk}

In particular case that all $a_i$ belong to $K \cup \{\infty\}$, then the isomorphism $C \cong C_{\lambda,\mu,\nu}$ which we constructed is defined over $K$ if and only if $\kappa$ is a square in $K$.
In short, we have the following lemma.
\begin{Lem}\label{defined}
Consider the genus-2 curve $C$ given as (\ref{sextic}), where each $a_i$ is a distinct element of $K \cup \{\infty\}$.
Then the curve $C_{\lambda,\mu,\nu}$ is a Rosenhain form of $C$, where $\lambda,\mu,\nu$ are defined as (\ref{lmn}).
In addition, the curve $C$ is isomorphic to $C_{\lambda,\mu,\nu}$ over $K$ if and only if $\kappa$ is a square in $K$, where $\kappa$ is defined as (\ref{epsilon}).
\end{Lem}
\begin{proof}
This is a direct result of the above discussion.
\end{proof}

At the last of this subsection, we prepare the following lemma for the proofs of our Main Theorems.
\begin{Lem}\label{alltoone}
Assume that all the 9 values
\[
    \lambda,\,\mu,\,\nu,\,1-\lambda,\,1-\mu,\,1-\nu,\,\lambda-\mu,\,\mu-\nu,\,\nu-\lambda
\]
are squares in $K$.
Then, the transformation from $C_{\lambda,\mu,\nu}$ to any Rosenhain form of $C_{\lambda,\mu,\nu}$ is defined over $K$.
\end{Lem}
\begin{proof}
Let $\{a_1,a_2,a_3,a_4,a_5,a_6\} = \{0,1,\infty,\lambda,\mu,\nu\}$ and $c=1$ in the equation (\ref{sextic}), then we have $C = C_{\lambda,\mu,\nu}$ clearly.
By the assumption, all $a_i-a_j$ are squares in $K$ for $i,j \in \{1,\ldots,6\}$.
Then $\kappa$ is also a square in $K$ for any choice $(a_1,a_2,a_3)$, and thus all transformations are defined over $K$ by Lemma \ref{defined}.
\end{proof}

\subsection{Reduced automorphism groups of genus-2 curves}\label{RA}
From now on, we have the following notations for the groups:
\begin{itemize}
    \item Let ${\rm C}_n\,(\cong \mathbb{Z}/n\mathbb{Z})$ be the cyclic group of order $n$.
    \item Let ${\rm D}_{2n}\hspace{-0.3mm}$ be the dihedral group of order $2n$.
    \item Let ${\rm S}_n\hspace{-0.3mm}$ be the symmetric group of degree $n$.
\end{itemize}
First of all, we recall the definition of the reduced automorphism group of a hyperelliptic curve.
\begin{Def}
For a hyperelliptic curve $C$, the {\it reduced automorphism group} of $C$ is defined as
\[
    {\rm RA}(C) := {\rm Aut}(C)/\langle\iota\rangle,
\]
where ${\rm Aut}(C)$ denotes the automorphism group of $C$ over $\sbar{K}$, and $\iota$ is the hyperelliptic involution. 
\end{Def}

Igusa \cite[Section 8]{Igusa} classified all genus-2 curves $C$ by their reduced automorphism groups and gave explicit equations of a Rosenhain forms of $C$ (see also Ibukiyama-Katsura-Oort \cite[Section 1.2]{IKO}).
On the other hand, Katsura-Takashima \cite[Section 5]{KT} gave other equations of $C$ (we call them {\it normal forms} of $C$).
The following is a table summarizing their results for $p \geq 7$; there are 7 possible reduced automorphism groups.
We denote by $i$ a square root of $-1$ and by $\zeta$ a primitive fifth root of unity.\vspace{-1mm}

\begin{center}
\begin{table}[h]
\begin{tabular}{|c|c|c|}\hline
     ${\rm RA}(C)$ & A normal form of $C$ & A Rosenhain form of $C$\\\hline
     $\{1\}$ & --------- & $y^2 = x(x-1)(x-\lambda)(x-\mu)(x-\nu)$\\[0.5mm]
     ${\rm C}_2$ & $y^2 = (x^2-1)(x^2-a)(x^2-b)$ & $y^2=x(x-1)(x-\lambda)(x-\mu)\bigl(x-\frac{\lambda(1-\mu)}{1-\lambda}\bigr)$ \\[1mm]
     ${\rm S}_3$ & $y^2 = (x^3-1)(x^3-a)$ & $y^2=x(x-1)(x-\lambda)(x-\frac{\lambda-1}{\lambda})\bigl(x-\frac{1}{1-\lambda}\bigr)$\\[1mm]
     $({\rm C}_2)^2$ & $y^2 = x(x^2-1)(x^2-a)$ & $y^2 = x(x-1)(x+1)(x-\lambda)\bigl(x-\frac{1}{\lambda}\bigr)$\\[1mm]
     ${\rm D}_{12}$ & $y^2 = x^6-1$ & $y^2 = x(x-1)(x+1)(x-2)\bigl(x-\frac{1}{2}\bigr)$\\[0.5mm]
     ${\rm S}_4$ & $y^2 = x^5-x$ & $y^2 = x(x-1)(x+1)(x-i)(x+i)$ \\[0.5mm]
     ${\rm C}_5$ & $y^2 = x^5-1$ & $y^2 = x(x-1)(x-1-\zeta)(x-1-\zeta-\zeta^2)(x-1-\zeta-\zeta^2-\zeta^3)$\\\hline
\end{tabular}
\[
    \begin{tikzcd}[column sep = small]
        && \{1\} \arrow[ld,dash]\arrow[rrrrrddd,dash] &&&&&\\
        & {\rm C}_2 \arrow[ld,dash]\arrow[rd,dash] &&&&&&\\
        \hspace{-2.5mm}({\rm C}_2)^2 \arrow[rd,dash] && {\rm S}_3 \arrow[ld,dash]\arrow[rd,dash] &&&&&\\
        & {\rm D}_{12} && \hspace{2mm}{\rm S}_4 &&&& {\rm C}_5
    \end{tikzcd}\vspace{-3mm}
\]
\caption{The possible reduced automorphism groups of genus-2 curves}
\end{table}
\end{center}\vspace{-8mm}

\begin{Rmk}\label{p35}
For $p=3$, the case that ${\rm RA}(C) \cong {\rm D}_{12}$ disappears since $C: y^2 = x^6 - 1$ has a singular point.
For $p=5$, the last three cases reduce the case that ${\rm RA}(C) \cong {\rm PGL}(2,5)$ according to \cite[p.\,645]{Igusa}.
\end{Rmk}
\begin{Rmk}
There is a criterion to determine the reduced automorphism group of a genus-2 curve by using the Clebsch invariants (cf. \cite[Section 3.2]{FS}).
\end{Rmk}

Not all of the classifications in the table above are necessary, but we will only use the following result.
\begin{Prop}\label{classification}
Assume that $p \geq 3$. Let $C$ be a genus-2 curve, then the following statements are true:
\begin{enumerate}
\item If ${\rm RA}(C) \cong {\rm C}_5$, then $C$ is isomorphic to the curve defined by the equation $y^2 = x^5-1$.
\item If ${\rm RA}(C) \supset {\rm C}_2$, then $C$ is isomorphic to the curve defined by the equation
\begin{equation}\label{split}
    y^2 = (x^2-1)(x^2-a)(x^2-b)
\end{equation}
for some $a,b$ (belonging to $\sbar{K}$).
\item Otherwise, the reduced automorphism group of $C$ is trivial (i.e.\ \,${\rm RA}(C) = \{1\}$).
\end{enumerate}
\end{Prop}
\begin{proof}
(1) This follows from \cite[Section 8]{Igusa} or \cite[Section 1.2]{IKO}.\par\vspace{0.5mm}
\hspace{5.6mm}(2) This follows from \cite[Lemma 2.2]{GS}.\par\vspace{0.5mm}
\hspace{5.6mm}(3) This assertion holds since no other reduced automorphism groups of $C$ exist except $\{1\}$, from the above table and Remark \ref{p35}.
\end{proof}

For the second case in Proposition \ref{classification}, there exist two involutions $\sigma,\tau$ defined by
\begin{align*}
    \sigma: C \rightarrow C\,&;\,(x,y) \mapsto (-x,y),\\
    \tau: C \rightarrow C\,&;\,(x,y) \mapsto (-x,-y).\\[-4.5mm]
\end{align*}
We put the quotients $E_1 := C/\langle\sigma\rangle$ and $E_2 := C/\langle\tau\rangle$, then they are elliptic curves defined by
\begin{alignat*}{2}
    E_1&: Y^2 = (X-1)(X-a)(X-b) & {\rm \ \,with\ } & X = x^2,\,Y = y,\\
    E_2&: Y^2 = X(X-1)(X-a)(X-b) & {\rm \,with\ } & X = x^2,\,Y = xy.
\end{alignat*}
Moreover, we see that these two morphisms $C \rightarrow E_1$ and $C \rightarrow E_2$ induce a $(2,2)$-isogeny ${\rm Jac}(C) \rightarrow E_1 \times E_2$. Conversely, it is known that if ${\rm Jac}(C)$ is $(2,2)$-isogenous to the product $E_1 \times E_2$ for elliptic curves $E_i$, then the curve $C$ coincided with the second case in Proposition \ref{classification} (i.e.\ \,${\rm RA}(C) \supset {\rm C}_2$) by \cite[Proposition\,4.3]{KT}.

\begin{Lem}\label{decomposed}
The curve $C$ of the form (\ref{split}) is superspecial if and only if there exist $t_1$ and $t_2$ such that
\[
    a = \frac{t_1}{t_2} \cdot \frac{1-t_2}{1-t_1}, \quad b = \frac{t_1}{t_2},
\]
and two elliptic curves $E_i: v^2 = u(u-1)(u-t_i)$ are supersingular.
\end{Lem}
\begin{proof}
As mentioned in the above discussion, there exists a $(2,2)$-isogeny $\phi: {\rm Jac}(C) \rightarrow E_1 \times E_2$ with
\begin{alignat*}{2}
    E_1&: Y^2 = (X-1)(X-a)(X-b),\\
    E_2&: Y^2 = X(X-1)(X-a)(X-b).\\[-5.5mm]
\end{alignat*}
By the same coordinate change as \cite[Section 2]{KO}, these elliptic curves are transformed into
\begin{alignat}{2}
    E_1&: (1-a)v^2 = u(u-1)\biggl(u-\frac{b-a}{1-a}\biggr)  &{\rm \ \,with\ \,}& u = \frac{X-a}{1-a},\ \,v = \frac{Y}{(1-a)^2},\label{E1}\\
    E_2&: b(1-a)v^2 = u(u-1)\biggl(u-\frac{b-a}{b(1-a)}\biggr) &{\rm \ \,with\ \,}& u = \frac{X-a}{X(1-a)},\ \,v = \frac{aY}{b(1-a)^2X^2}\label{E2}.
\end{alignat}
Hence, the curve $C$ is superspecial if and only if two elliptic curves $E_i: v^2 = u(u-1)(u-t_i)$ for $i = 1,2$ are supersingular
with
\[
    t_1 = \frac{b-a}{1-a}, \quad t_2 = \frac{b-a}{b(1-a)}.
\]
Solving these equations for $a$ and $b$, we have this lemma.
\end{proof}

\subsection{Richelot isogeny}\label{Richelot}
A {\it Richelot isogeny} is a $(2,2)$-isogeny whose domain is the Jacobian variety of a genus-2 curve.
Let us briefly recall the abstract description of this according to \cite[Section 3]{KT}, which helps readers understand.
Let $C$ be a genus-2 curve.
We denote by $J := {\rm Jac}(C)$ be the Jacobian variety of $C$, and denote by $J^t$ the dual abelian variety of $J$.
Considering $C$ as a divisor of $J$, it defines a principal polarization $\varphi_C: J \cong J^t$.
Therefore, the divisor $2C$ defines a polarization $\varphi_{2C} : J \to J^t$, whose kernel is equal to $J[2]$.
For each isotropic subgroup $G$ of $J[2]$ (we remark that $G$ is isomorphic to $\mathbb{Z}/2\mathbb{Z} \times \mathbb{Z}/2\mathbb{Z}$), we have a Richelot isogeny $\pi: J \to J/G$.
It follows from \cite[Section 23, Corollary of Theorem 2]{AV} that there exists a divisor $C'$ on $J/G$ such that $\pi^*(C')=2C$.
One can show that $C'$ defines a principal polarization on $J/G$.
By the construction (based on descent theory as in \cite[Section 12]{AV}), we see that $C'$ is defined over $K$ if both $C$ and $G$ are defined over $K$.

Richelot isogenies are useful in enumerating superspecial genus-2 curves. In fact, the following theorem by Jordan-Zaytman \cite[Section 7]{JZ} is known:
\begin{Thm}\label{connected}
For any $p \geq 3$, the superspecial $(2,2)$-isogeny graph $\mathcal{G}_p$ is connected.
\end{Thm}
\noindent Here, the superspecial $(2,2)$-isogeny graph $\mathcal{G}_p$ is defined as follows:
The vertices of $\mathcal{G}_p$ are isomorphism classes of superspecial principally polarized abelian surfaces defined over $\mathbb{F}_{p^2}$.
The edges of $\mathcal{G}_p$ are isomorphism classes of $(2,2)$-isogenies between superspecial principally polarized abelian surfaces.
Then, thanks to Theorem \ref{connected}, we can enumerate all superspecial genus-2 curves using the following algorithm (cf. \cite[Algorithm 7.1]{KHH}).
\newpage
\begin{algorithm}
\caption{\ Calculating superspecial genus-2 curves using Richelot isogenies.}
\begin{algorithmic}[1]\label{Algo}
\REQUIRE A rational prime $p \geq 7$.
\ENSURE A list $\mathcal{L}$ of all superspecial genus-2 curves over $\mathbb{F}_{p^2}$.
\STATE Compute the set ${\rm SsgEll}(p^2)$ of $\mathbb{F}_{p^2}$-isomorphism classes of supersingular elliptic curves over $\mathbb{F}_{p^2}$.
\STATE Set $\mathcal{L} \leftarrow \emptyset$.
\STATE For each pair $(E,E')$ of elements in ${\rm SsgEll}(p^2)$, compute the curves $C$ whose Jacobians are $(2,2)$-isogenous to $E \times E'$ (see \cite[Section\,3]{HLP}). If $C$ is not isomorphic to an element of $\mathcal{L}$, then adjoin it to $\mathcal{L}$.
\STATE Write $\mathcal{L} = \{C_1,\ldots,C_n\}$, and set $i \leftarrow 1$.
\STATE Compute the genus-2 curves $C'$ which are Richelot isogenous to $C_i$. If $C'$ is not isomorphic to an element of $\mathcal{L}$, then set $N \leftarrow \#\mathcal{L},\ C_{N+1} \leftarrow C'$ and adjoin it to $\mathcal{L}$.
\STATE If $i < \#\mathcal{L}$, then set $i \leftarrow i+1$ and go back to Step\,5.
\RETURN $\mathcal{L}$.
\end{algorithmic}
\end{algorithm}\vspace{-3mm}
\begin{Rmk}
As mentioned in Subsection \ref{RA}, the Jacobian variety of genus-2 curve $C$ is $(2,2)$-isogenous to the product of two elliptic curves if and only if ${\rm RA}(C) \supset {\rm C}_2$.
Hence, we see that all the curves $C$ generated in Step 3 satisfy ${\rm RA}(C) \supset {\rm C}_2$.
\end{Rmk}

Next, let us review how to compute the genus-2 curve which is Richelot isogenous to a given genus-2 curve (see \cite[Section 3.2]{CDS} for details).
\begin{Def}
Let $f(X) \in K[X]$ be a separable polynomial of degree $5$ or $6$. Then, a {\it quadratic splitting} of $f(X)$ is a set $\hspace{-0.2mm}\{G_1,G_2,G_3\} \subset\hspace{-0.2mm} \sbar{K}[x]$ of three monic polynomials of degree 1 or 2 such that $G_1G_2G_3 = f(X)$.
\end{Def}
Consider the genus-2 curve
\[
    C: Y^2 = (X-a_1)(X-a_2)(X-a_3)(X-a_4)(X-a_5)(X-a_6) =: f(X)
\]
where each $a_i$ is a distinct element of $ \sbar{K} \cup \{\infty\}$.
For $i \in \{1,\ldots,6\}$, we put $P_i := \infty$ if $a_i = \infty$ and $P_i := (a_i,0)$ otherwise.
Then, all 2-torsion points on $J := {\rm Jac}(C)$ are written as
\[
    D_{i,j} := [P_i] - [P_j]\ \,{\rm with}\ \,i < j.
\]
Now, we can construct the Richelot isogeny $\pi: J \rightarrow J/G$ with kernel $G = \{0,D_{1,2},D_{3,4},D_{5,6}\} \cong \mathbb{Z}/2\mathbb{Z} \times \mathbb{Z}/2\mathbb{Z}$ as follows.
\begin{Prop}\label{formula}
With the notations as above, let $\{G_1,G_2,G_3\}$ be a quadratic splitting of $f(X)$ with
\[
    \left\{
    \begin{array}{l}
        G_1 = g_{1,2}X^2 + g_{1,1}X + g_{1,0} = (X-a_1)(X-a_2),\\[1mm]
        G_2 = g_{2,2}X^2 + g_{2,1}X + g_{2,0} = (X-a_3)(X-a_4),\\[1mm]
        G_3 = g_{3,2}X^2 + g_{3,1}X + g_{3,0} = (X-a_5)(X-a_6).
    \end{array}
    \right.
\]
Remark that if $\deg{f(X)} = 5$, then one of $\{a_1,\ldots,a_6\}$ becomes $\infty$ and one of $\{g_{1,2},g_{2,2},g_{3,2}\}$ is equal to zero.
Setting
\[
    \delta := \det\hspace{-0.3mm}\begin{pmatrix}
        g_{1,2} & g_{1,1} & g_{1,0}\\
        g_{2,2} & g_{2,1} & g_{2,0}\\
        g_{3,2} & g_{3,1} & g_{3,0}
    \end{pmatrix}\!,
\]
we have the following statements:
\begin{enumerate}
\item If $\delta \neq 0$, then $J/G$ is isomorphic to the Jacobian of the genus-2 curve $C': y^2 = \delta^{-1}H_1H_2H_3$ where
    \[
        H_1 := {G'}_{\!2}G_3-G_2{G'}_{\!3}, \quad H_2 := {G'}_{\!3}G_1-G_3{G'}_{\!1}, \quad H_3 := {G'}_{\!1}G_2-G_1{G'}_{\!2}.
    \]
Moreover, if all 2-torsion points on $J$ are defined over $K$, then any isotropic subgroup of $J[2]$ is defined over $K$.
Hence, the isogeny is also defined over $K$.
\item If $\delta = 0$, then $J/G$ is isomorphic to a product of two elliptic curves.
\end{enumerate}
\end{Prop}
\begin{proof}
See \cite[Chapter 8]{Smith}.
\end{proof}

\newpage
\section{Proof of Main Theorems}
In this section, we show our Main Theorems stated in Section 1 (specifically, the proofs of Main Theorem A and B are given in Subsection 3.1 and 3.2, respectively). We use the same notations as in previous sections, and we consider all curves over a field of characteristic $p \geq 3$.

\subsection{Proof of Main Theorem A}\label{A}
First of all, we show the following proposition (a partial result of Main Theorem A).
\begin{Prop}\label{lie}
Assume that the genus-2 curve
\[
    C: y^2 = x(x-1)(x-\lambda)(x-\mu)(x-\nu)
\]
is superspecial. Then $\lambda,\mu$ and $\nu$ belong to $\mathbb{F}_{p^2}$.
\end{Prop}
\begin{proof}
We divide into three cases by the reduced automorphism group of $C$ as in Proposition \ref{classification}.\vspace{-1mm}
\begin{enumerate}
\item The case that ${\rm RA}(C) \cong {\rm C}_5$: Recall that the curve $C$ is isomorphic to
\[
    C: y^2 = x^5 - 1 = (x-1)(x-\zeta)(x-\zeta^2)(x-\zeta^3)(x-\zeta^4),
\]
where $\zeta$ denotes a primitive fifth root of unity. It suffices to show that $\zeta$ is an element of $\mathbb{F}_{p^2}\!$ when $C$ is superspecial, since all Rosenhain invariants of $C$ can be written as a fractional expression of $\zeta$. It is well-known \cite[Proposition 1.13]{IKO} that this curve is superspecial if and only if $p \equiv 4 \pmod{5}$, and hence one can check that $\zeta^{p^2}\! = \zeta$ when $p \equiv 4 \pmod{5}$.\vspace{-1mm}
\item The case that ${\rm RA}(C) \supset {\rm C}_2$: Recall from Proposition \ref{classification} that $C$ is isomorphic to
\[
    C: y^2 = (x^2-1)(x^2-a)(x^2-b),
\]
where $a$ and $b$ belong to the algebraic closure of ${\mathbb F}_{p^2}$. Using Lemma \ref{decomposed}, there exist $t_1$ and $t_2$ such that
\[
    a = \frac{t_1}{t_2} \cdot \frac{1-t_2}{1-t_1}, \quad b = \frac{t_1}{t_2},
\]
and two elliptic curves $E_i: v^2 = u(u-1)(u-t_i)$ are supersingular. Here, it is known \cite[Proposition\,3.1]{AT} that $t_i$ and $1-t_i$ for $i = 1,2$ are fourth powers in $\mathbb{F}_{p^2}$. Thereby, we obtain that $a$ and $b$ are also fourth powers in $\mathbb{F}_{p^2}$. This implies that we can write
\[
    C: y^2 = (x-1)(x+1)(x-\!\sqrt{a})(x+\!\sqrt{a})(x-\!\sqrt{b})(x+\!\sqrt{b})
\]
with $\!\sqrt{a},\sqrt{b} \in \mathbb{F}_{p^2}$. Then $\lambda,\mu$ and $\nu$ obtained by transforming to a Rosenhain form of $C$ belong to $\mathbb{F}_{p^2}$, as mentioned in Subsection\,\ref{Rosenhain}.\vspace{-1mm}
\item The case that ${\rm RA}(C) = \{1\}$: It is well-known \cite[p.\,166]{Ekedahl} that $C$ descends to a maximal curve $C'$ defined over $\mathbb{F}_{p^2}$, where the square $F^2$ of the Frobenius map $F$ is equal to $-p$. Let $C'$ be written as
\[
    C': Y^2 = \kappa (X-a_1)(X-a_2)(X-a_3)(X-a_4)(X-a_5)(X-a_6) =: f(X)
\]
where $f(X) \in \mathbb{F}_{p^2}[X]$ is a square-free polynomial of degree $6$ and $\kappa$ belongs to $\mathbb{F}_{p^2}$. Then $P_i := (a_i,0)$ is a Weierstrass point of $C$, and moreover $D_{i,j} := [P_i] - [P_j]$ with $i \neq j$ is a 2-torsion point on $J := {\rm Jac}(C)$. In the following, we show that $a_i -a_j \in \mathbb{F}_{p^2}\hspace{-0.3mm}$ for all $i,j \in \{1,\ldots,6\}$. Indeed, we can choose $k \in \{1,\ldots,6\}$ such that $k \neq i$ and $k \neq j$. Then $D_{i,k} \in J$ is defined over $\mathbb{F}_{p^2}$ since $F^2D_{i,k} = -pD_{i,k} = D_{i,k}$. Here, the Mumford representation (cf. \cite[Section 13]{Washington}) for $D_{i,k}$ is given as $(u_{i,k},0)$ with
\[
    u_{i,k} := (t-a_i)(t-a_k) = t^2 - (a_i+a_k)t + a_ia_k \in \mathbb{F}_{p^2}[t],
\]
and hence $a_i + a_k$ belongs to $\mathbb{F}_{p^2}$. Since $a_j + a_k$ belongs to $\mathbb{F}_{p^2}\hspace{-0.3mm}$ similarly, we obtain that $a_i-a_j \in \mathbb{F}_{p^2}$. Then, as studied in Subsection \ref{Rosenhain}, for all Rosenhain forms
\[
    C': \kappa' y^2 = x(x-1)(x-\lambda)(x-\mu)(x-\nu), \quad \kappa' \in \mathbb{F}_{p^2},
\]
we see that $\lambda,\mu$ and $\nu$ also belong to $\mathbb{F}_{p^2}$ since these are obtained as $a_i-a_j$'s quotients. Now, we denote by $G := {\rm Gal}(\mbar{{\mathbb F}_{p^2}\!}/\mathbb{F}_{p^2}\hspace{-0.3mm})$, then it is known \cite[Section 4]{Serre} that there is a bijection from the set of $\mathbb{F}_{p^2}$-forms of $C'$ to $H^1(G,{\rm Aut}(C'))$, which is isomorphic to $\mathbb{Z}/2\mathbb{Z}$ by the assumption. Hence $C$ is isomorphic to
\[
    \kappa'y^2 = x(x-1)(x-\lambda)(x-\mu)(x-\nu)
\]
for $\kappa'=1$ or $\varepsilon$ with a non-square element $\varepsilon$ in $\mathbb{F}_{p^2}$. In any case, all 
$\lambda,\mu$ and $\nu$ belong to $\mathbb{F}_{p^2}\hspace{-0.3mm}$ as desired.
\end{enumerate}
Therefore, the proof is done.
\end{proof}

Next, given a Rosenhain form of a superspecial genus-2 curve $C$, we compute Rosenhain forms of $C'$ which are Richelot isogenous to $C$, according to Proposition \ref{formula}. We consider the superspecial genus-2 curve
\[
    C: Y^2 = (X-a_1)(X-a_2)(X-a_3)(X-a_4)(X-a_5)
\]
with $\{a_1,a_2,a_3,a_4,a_5\} = \{0,1,\lambda,\mu,\nu\}$ and a quadratic splitting\vspace{-0.5mm}
\begin{align*}
    G_1 &:= X-a_1,\\
    G_2 &:= (X-a_2)(X-a_3),\\
    G_3 &:= (X-a_4)(X-a_5).
\end{align*}
We remark that all $a_i$ belong to $\mathbb{F}_{p^2}\!$ by Proposition \ref{lie}. Moreover, we define the following three values\vspace{-0.5mm}
\begin{align}\label{discriminant}
    D_1 &:= (a_2-a_4)(a_2-a_5)(a_3-a_4)(a_3-a_5),\nonumber\\
    D_2 &:= (a_1-a_4)(a_1-a_5),\\
    D_3 &:= (a_1-a_2)(a_1-a_3),\nonumber\\[-6.5mm]\nonumber
\end{align}
then we can compute\vspace{-0.5mm}
\[
    \delta = \det\!\begin{pmatrix}
        0 & 1 & -a_1\\
        1 & -a_2-a_3 & a_2a_3\\
        1 & -a_4-a_5 & a_4a_5
    \end{pmatrix} = -a_1a_2 - a_1a_3 + a_1a_4 + a_1a_5 + a_2a_3 - a_4a_5 = -(D_2 - D_3).\vspace{0.5mm}
\]
In the following, we choose $\!\sqrt{D_1},\!\sqrt{D_2}$ and $\!\sqrt{D_3}$ of a square root of $D_1,D_2$ and $D_3$ (these values are defined to be elements of $\mbar{{\mathbb F}_{p^2}}$, but they will turn out to be elements of $\mathbb{F}_{p^2}\hspace{-0.3mm}$). Then, three polynomials $H_1,H_2$ and $H_3$ defined in Proposition \ref{formula} can be calculated as
\begin{align*}
    H_1 &= (a_2+a_3-a_4-a_5)X^2 - 2(a_2a_3 - a_4a_5)X + a_2a_3a_4 + a_2a_3a_5 - a_2a_4a_5 - a_3a_4a_5\\
    &= (a_2+a_3-a_4-a_5)(X-\alpha_1)(X-\alpha_2),\\
    H_2 &= x^2 - 2a_1x + a_1a_4 + a_1a_5 - a_4a_5 = (X-\beta_1)(X-\beta_2),\\
    H_3 &= -x^2 + 2a_1x - a_1a_2 - a_1a_3 + a_2a_3 = -(X-\gamma_1)(X-\gamma_2),\\[-6mm]
\end{align*}
where we define $\alpha_i,\beta_i$ and $\gamma_i$ for $i \in \{1,2\}$ to be elements of $\mbar{{\mathbb F}_{p^2}}$ as follows:\vspace{-0.5mm}
\begin{align*}
    \alpha_1 = \frac{(a_2a_3 - a_4a_5) + \hspace{-0.3mm}\sqrt{D_1}}{a_2+a_3-a_4-a_5}, &\quad \alpha_2 = \frac{(a_2a_3 - a_4a_5) - \hspace{-0.3mm}\sqrt{D_1}}{a_2+a_3-a_4-a_5},\nonumber\\
    \beta_1 = a_1 + \hspace{-0.3mm}\sqrt{D_2}, &\quad \beta_2 = a_1 - \hspace{-0.3mm}\sqrt{D_2},\\
    \gamma_1 = a_1 + \hspace{-0.3mm}\sqrt{D_3}, &\quad \gamma_2 = a_1 - \hspace{-0.3mm}\sqrt{D_3}.\nonumber\\[-5.5mm]
\end{align*}
As mentioned in Subsection \ref{Richelot}, the genus-2 curve
\begin{align}\label{Cdash}
    C': y^2 &= \delta^{-1}H_1H_2H_3\nonumber\\
    &= c(X-\alpha_1)(X-\alpha_2)(X-\beta_1)(X-\beta_2)(X-\gamma_1)(X-\gamma_2), \quad c := -(a_2+a_3-a_4-a_5)\delta^{-1}
\end{align}
is Richelot isogenous to $C$, and hence $C'$ is also superspecial by the assumption. The transformation
\[
    X \mapsto \frac{X-\gamma_1}{X-\gamma_2} \cdot \frac{\beta_1-\gamma_2}{\beta_1-\gamma_1} =: x, \ \, Y \mapsto \frac{c^{-1}Y}{(\gamma_2-\gamma_1)(\gamma_2-\beta_1)(\gamma_2-\beta_2)(\gamma_2-\alpha_1)(\gamma_2-\alpha_2)}\biggl(\frac{\gamma_2-\gamma_1}{X-\gamma_2} \cdot \frac{\beta_1-\gamma_2}{\beta_1-\gamma_1}\biggr)^{\!\!3} =: y
\]
gives a Rosenhain form\vspace{1mm}
\[
    C': \kappa y^2 = x(x-1)(x-\lambda')(x-\mu')(x-\nu')
\]
with\vspace{-1.5mm}
\begin{align*}
    \lambda' &:= \biggl(\frac{\sqrt{D_2}+\!\sqrt{D_3}}{\sqrt{D_2}-\!\sqrt{D_3}}\biggr)^{\!\!2} = \frac{(D_2+D_3) + 2\sqrt{D_2D_3}}{(D_2+D_3) - 2\sqrt{D_2D_3}},\\
    \mu' &:= \frac{(a_2a_3 - a_4a_5) + \sqrt{D_1} - (a_2+a_3-a_4-a_5)\bigl(a_1+\!\sqrt{D_3}\bigr)}{(a_2a_3 - a_4a_5) + \sqrt{D_1} - (a_2+a_3-a_4-a_5)\bigl(a_1-\!\sqrt{D_3}\bigr)} \cdot \frac{\sqrt{D_2}+\!\sqrt{D_3}}{\sqrt{D_2}-\!\sqrt{D_3}},\\
    \nu' &:= \frac{(a_2a_3 - a_4a_5) - \sqrt{D_1} - (a_2+a_3-a_4-a_5)\bigl(a_1+\!\sqrt{D_3}\bigr)}{(a_2a_3 - a_4a_5) - \sqrt{D_1} - (a_2+a_3-a_4-a_5)\bigl(a_1-\!\sqrt{D_3}\bigr)} \cdot \frac{\sqrt{D_2}+\!\sqrt{D_3}}{\sqrt{D_2}-\!\sqrt{D_3}}
\end{align*}
and
\[
    \kappa := -\delta^{-1}(a_2+a_3-a_4-a_5)(\gamma_1-\beta_1)(\gamma_2-\beta_2)(\gamma_2-\alpha_1)(\gamma_2-\alpha_2).
\]
Then $\lambda',\mu'$ and $\nu'$ are elements of $\mathbb{F}_{p^2}\!$ by using Proposition \ref{lie} again. One can check that
\[
    \frac{\lambda'+1}{\lambda'-1} = \frac{D_2+D_3}{2\sqrt{D_2D_3}} \in \mathbb{F}_{p^2},
\]
and hence $D_2D_3$ is a square in $\mathbb{F}_{p^2}$. Hence
\[
    \sqrt{\lambda'} := \frac{\sqrt{D_2}+\!\sqrt{D_3}}{\sqrt{D_2}-\!\sqrt{D_3}} = \frac{(D_2+D_3) + 2\sqrt{D_2D_3}}{D_2-D_3}
\]
is an element of $\mathbb{F}_{p^2}$. Tedious computation shows that
\[
    \sqrt{D_3} = -\frac{D_2 - D_3}{a_2+a_3-a_4-a_5} \cdot \frac{(\sqrt{\lambda'}-\mu')(\sqrt{\lambda'}-\nu')}{\lambda'-\mu'\nu'} \in \mathbb{F}_{p^2},
\]
which implies that $D_3$ is a square in $\mathbb{F}_{p^2}$\,(and hence $D_2$ is also a square in $\mathbb{F}_{p^2}$).
%\begin{Rmk}
%The fact that all $D_1,D_2$ and $D_3$ are squares in $\mathbb{F}_{p^2}$ can also be showed more theoretically as follows. Let $D'_{i,j} = [P'_{\,i}] - [P'_{\,j}]$ be $2$-torsion points on $C'$ with $P'_{\,i} \in C'$. As $C'$ is superspecial, the square $F^2$ of the Frobenius map $F$ acts on ${\rm Jac}(C')[2]$ by multiplication by $-p$. Then, we obtain $F^2D'_{ij} = -pD'_{ij} = D'_{ij}$. Using the Riemann-Roch theorem, the points $P'_i$ are fixed by $\sigma^2$ (to be checked). Hence $H_i$ are decomposed into polynomials of degree one over ${\mathbb F}_{p^2}$.
%\end{Rmk}
\begin{proof}[\underline{Proof of Main Theorem A\,(1)}]
We can take $(a_1,a_2,a_3)$ arbitrarily so that $\{a_1,a_2,a_3,a_4,a_5\} = \{0,1,\lambda,\mu,\nu\}$. For example if we take $(a_1,a_2,a_3) = (0,1,\lambda)$, then we obtain $D_3 = \lambda$, which turns out to be a square in $\mathbb{F}_{p^2}$ from the above discussion. Similarly we can show that other 8 values are also squares in $\mathbb{F}_{p^2}$.
\end{proof}

Before the proof of the second assertion of Main Theorem A, we show the following lemma:
\begin{Lem}\label{fourth}
With notations as above, three values $\!\sqrt{D_1D_2},\sqrt{D_2D_3}$ and $\sqrt{D_3D_1}$ are all squares in $\mathbb{F}_{p^2}$.
\end{Lem}
\begin{proof}
Recall from the first assertion of Main Theorem A that $a_i-a_j$ are all squares for all $i,j \in \{1,\ldots,5\}$, and hence $D_1,D_2$ and $D_3$ in (\ref{discriminant}) are all squares in $\mathbb{F}_{p^2}$. This means that $\!\sqrt{D_1},\sqrt{D}_2$ and $\!\sqrt{D_3}$ are all elements in $\mathbb{F}_{p^2}$. Then, one can compute that\vspace{-2mm}
\begin{align*}
    1-\lambda' = -\frac{4\sqrt{D_2D_3}}{\bigl(\!\sqrt{D_2}-\!\sqrt{D_3}\bigr)^{\!2}},
\end{align*}
and hence $\hspace{-0.3mm}\sqrt{D_2D_3}$ is a square in $\mathbb{F}_{p^2}$. Moreover, tedious computation shows that
\[
    \mu' - \nu' = \frac{4\sqrt{D_3D_1}}{\bigl(\!\sqrt{D_2}-\!\sqrt{D_3}\bigr)^{\!2}\!\bigl(2a_1-a_2-a_3-2\sqrt{D_3}\bigr)}.
\]
Here, we let $b_{12}$ and $b_{13}$ elements of $\mathbb{F}_{p^2}\!$ such that $(b_{12})^2 = a_1 - a_2,\ (b_{13})^2 = a_1 - a_3$ and $b_{12}b_{13} = \!\sqrt{D_3}$, then we obtain $2a_1-a_2-a_3-2\sqrt{D_3} = (b_{12}-b_{13})^2$. This implies that $2a_1-a_2-a_3-2\sqrt{D_3}$ is a square in $\mathbb{F}_{p^2}$, and $\!\sqrt{D_3D_1}$ is a square in $\mathbb{F}_{p^2}$. Since $\!\sqrt{D_1D_2}$ can computed from other two values and $D_3$, hence $\!\sqrt{D_1D_2}$ is also a square in $\mathbb{F}_{p^2}$.
\end{proof}
\begin{proof}[\underline{Proof of Main Theorem A\,(2)}]
It follows from Lemma \ref{fourth} that the value
\[
    D_2D_3 = (a_1-a_2)(a_1-a_3)(a_1-a_4)(a_1-a_5)
\]
is a fourth power in $\mathbb{F}_{p^2}$. For example, if we take $a_1=0$, then we obtain $D_2D_3 = \lambda\mu\nu$, which turns out to be a fourth power in $\mathbb{F}_{p^2}$. Similarly, one can show that other 4 values are also fourth powers in $\mathbb{F}_{p^2}$.
\end{proof}

\subsection{Proof of Main Theorem B}\label{B}
First, we show the following propositions (a partial result of Main Theorem B).
\begin{Lem}\label{maxminC2}
Assume that the genus-2 curve\vspace{-0.5mm}
\[
    C: Y^2 = (X^2-1)(X^2-a)(X^2-b)\vspace{-0.5mm}
\]
is superspecial. Then, the curve $C$ is maximal or minimal over $\mathbb{F}_{p^2}$. Moreover, we have the following:
\begin{itemize}
\item The case of $p \equiv 3 \pmod{4}$: The curve $C$ is maximal over $\mathbb{F}_{p^2}\hspace{-0.3mm}$ if and only if $1-a$ is a square in $\mathbb{F}_{p^2}$.
\item The case of $p \equiv 1 \pmod{4}$: 
The curve $C$ is maximal over $\mathbb{F}_{p^2}\hspace{-0.3mm}$ if and only if $1-a$ is not a square in $\mathbb{F}_{p^2}$.
\end{itemize}
\end{Lem}
\begin{proof}
Recall from Subsection \ref{RA} that the Jacobian variety of $C$ is $(2,2)$-isogenous to $E_1 \times E_2$ with\vspace{-0.5mm}
\begin{alignat*}{2}
    E_1&: (1-a)v^2 = u(u-1)(u-t_1)  &{\rm \ \,with\ \,}& t_1 := \frac{b-a}{1-a},\\
    E_2&: b(1-a)v^2 = u(u-1)(u-t_2) &{\rm \ \,with\ \,}& t_2 := \frac{b-a}{b(1-a)}.
\end{alignat*}
This isogeny is defined over $\mathbb{F}_{p^2}$. Indeed, recall from (\ref{E1}) and (\ref{E2}) that this is explicitly written by $a$ and $b$, and it follows from the proof of Proposition \ref{lie} that both $a$ and $b$ are squares in $\mathbb{F}_{p^2}\hspace{-0.3mm}$ from the superspeciality of $C$. This fact means that $C$ is maximal (resp. minimal) if and only if both $E_i$ are maximal (resp. minimal).
Here, Auer-Top's result \cite[Lemma\,2.2]{AT} shows that Legendre elliptic curves $y^2 = x(x-1)(x-t)$ are maximal (resp. minimal) over $\mathbb{F}_{p^2}\hspace{-0.3mm}$ if and only if $p \equiv 3$ (resp. $p \equiv 1$). 
Therefore, we divide into two cases depending on whether $p \equiv 3 \pmod{4}$ or $p \equiv 1 \pmod{4}$ in the following.
\begin{itemize}
    \item The case of $p \equiv 3 \pmod{4}$: Two elliptic curves $v^2 = u(u-1)(u-t_i)$ for $i = 1,2$ are maximal over $\mathbb{F}_{p^2}$. Assume that $1-a$ is a square (resp.\ non-square) in $\mathbb{F}_{p^2}$, then $b(1-a)$ is also a square (resp.\ non-square). This implies that both $E_1$ and $E_2$ are maximal (resp.\ minimal) over $\mathbb{F}_{p^2}$, and hence the curve $C$ is also maximal (resp.\ minimal) over $\mathbb{F}_{p^2}$.\vspace{-0.5mm}
    \item The case of $p \equiv 1 \pmod{4}$: Two elliptic curves $v^2 = u(u-1)(u-t_i)$ for $i = 1,2$ are minimal over $\mathbb{F}_{p^2}$. Assume that $1-a$ is a square (resp.\ non-square) in $\mathbb{F}_{p^2}$, then $b(1-a)$ is also a square (resp.\ non-square). This implies that both $E_1$ and $E_2$ are minimal (resp.\ maximal) over $\mathbb{F}_{p^2}$, and hence the curve $C$ is also minimal (resp.\ maximal) over $\mathbb{F}_{p^2}$.
\end{itemize}
Therefore, this lemma is true.
\end{proof}
\begin{Prop}\label{C2max}
Assume that the genus-2 curve
\[
    C: y^2 = x(x-1)(x-\lambda)(x-\mu)(x-\nu)
\]
is superspecial and ${\rm RA}(C) \supset {\rm C}_2$. Then, we have the following:
\begin{itemize}
\item The case of $p \equiv 3 \pmod{4}$: The curve $C$ is maximal over $\mathbb{F}_{p^2}$.
\item The case of $p \equiv 1 \pmod{4}$: The curve $C$ is minimal over $\mathbb{F}_{p^2}$.
\end{itemize}
\end{Prop}
\begin{proof}
Recall from Subsection \ref{RA} that the curve $C$ has a form\vspace{-0.5mm}
\[
    Y^2 = (X^2-1)(X^2-a)(X^2-b),\vspace{-0.5mm}
\]
where $a$ and $b$ are squares in $\mathbb{F}_{p^2}$. Here, it suffices to show this proposition holds for {\it one} Rosenhain form of it by Lemma \ref{alltoone} and Main Theorem A(1). Set $(a_1,a_2,a_3) := (1,-1,\sqrt{a})$ and $\{a_4,a_5,a_6\} = \bigl\{-\sqrt{a},\sqrt{b},-\sqrt{b}\bigr\}$, then the transformation\vspace{-0.5mm}
\[
    X \mapsto \frac{X-a_1}{X-a_3} \cdot \frac{a_2-a_3}{a_2-a_1} =: x, \ \, Y \mapsto \frac{c^{-1}Y}{(a_3-a_1)(a_3-a_2)(a_3-a_4)(a_3-a_5)(a_3-a_6)}\biggl(\frac{a_3-a_1}{X-a_3}\cdot\frac{a_2-a_3}{a_2-a_1}\bigg)^{\!\!3} =: y\vspace{-0.5mm}
\]
gives the equation\vspace{1.5mm}
\begin{equation}\label{C2rosenhain}
    \kappa y^2 = x(x-1)(x-\lambda)(x-\mu)(x-\nu), \quad \kappa := 4\sqrt{a}(a-b)
\end{equation}
where $\lambda,\mu,\nu$ are given in (\ref{lmn}). Here, this $\kappa$ is a square in $\mathbb{F}_{p^2}\hspace{-0.3mm}$ if and only if $1-a$ is a square in $\mathbb{F}_{p^2}$. Indeed, recall from Lemma \ref{decomposed} that we can write
\[
    a = \frac{t_1}{t_2} \cdot \frac{1-t_2}{1-t_1}, \quad b = \frac{t_1}{t_2},
\]
where $t_i$ and $1-t_i$ are fourth powers in $\mathbb{F}_{p^2}$. Hence $\kappa = 4\sqrt{a}(a-b)$ is a square in $\mathbb{F}_{p^2}\hspace{-0.3mm}$ if and only if $a-b$ is a square in $\mathbb{F}_{p^2}$. Moreover $a-b$ is a square in $\mathbb{F}_{p^2}$ if and only if $1-a$ is a square in $\mathbb{F}_{p^2}$ since $a-b = -(1-a)t_1$. In the following, we divide into two cases depending
on whether $p \equiv 3 \pmod{4}$ or $p \equiv 1 \pmod{4}$.
\begin{itemize}
    \item The case of $p \equiv 3 \pmod{4}$: It follows from Lemma \ref{maxminC2} and the above discussion that the curve in (\ref{C2rosenhain}) is maximal if and only if $1-a$ is a square in $\mathbb{F}_{p^2}$. Since this condition is equivalent to that $\kappa$ is a square in $\mathbb{F}_{p^2}$, then the curve $C: y^2 = x(x-1)(x-\lambda)(x-\mu)(x-\nu)$ is maximal over $\mathbb{F}_{p^2}$.
    \item The case of $p \equiv 1 \pmod{4}$: It follows from Lemma \ref{maxminC2} and the above discussion that the curve in (\ref{C2rosenhain}) is minimal if and only if $1-a$ is a square in $\mathbb{F}_{p^2}$. Since this condition is equivalent to that $\kappa$ is a square in $\mathbb{F}_{p^2}$, then the curve $C: y^2 = x(x-1)(x-\lambda)(x-\mu)(x-\nu)$ is minimal over $\mathbb{F}_{p^2}$.
\end{itemize}
Therefore, this proposition is true.
\end{proof}
\begin{Prop}\label{isomax}
Assume that two genus-2 curves\vspace{-1mm}
\begin{align*}
    C&: y^2 = x(x-1)(x-\lambda)(x-\mu)(x-\nu),\\
    C'&: y^2 = x(x-1)(x-\lambda')(x-\mu')(x-\nu')\\[-6.5mm]
\end{align*}
are Richelot isogenous. If $C$ is maximal (resp. minimal) over $\mathbb{F}_{p^2}$, then so is $C'$.
\end{Prop}
\begin{proof}
Here, we use notations in Subsection \ref{A}. Thanks to Main Theorem A\,(1), all $\alpha_1,\alpha_2,\beta_1,\beta_2,\gamma_1$ and $\gamma_2$ are elements of $\mathbb{F}_{p^2}$. Hence, it follows from Proposition \ref{formula}\,(1) that a Richelot isogeny $\phi: {\rm Jac}(C) \rightarrow {\rm Jac}(C')$ is defined over $\mathbb{F}_{p^2}$, where $C'$ is the form in (\ref{Cdash}). The transformation\vspace{-2mm}
\[
    X \mapsto \frac{X-\alpha_2}{X-\alpha_1} \cdot \frac{\gamma_2-\alpha_1}{\gamma_2-\alpha_2} =: x, \ \, Y \mapsto \frac{Y}{(\gamma_2-\alpha_1)(\gamma_2-\alpha_2)(\gamma_2-\beta_1)(\gamma_2-\beta_2)(\gamma_2-\gamma_1)}\biggl(\frac{\alpha_1-\alpha_2}{X-\alpha_1} \cdot \frac{\gamma_2-\alpha_1}{\gamma_2-\alpha_2}\biggr)^{\!\!3} =: y
\]
gives a Rosenhain form\vspace{1mm}
\[
    C': \kappa y^2 = x(x-1)(x-\lambda')(x-\mu')(x-\nu'),
\]
for $\lambda',\mu',\nu' \in \mathbb{F}_{p^2}\hspace{-0.3mm}$ with
\begin{align*}
    \kappa &= -\delta^{-1}(a_2+a_3-a_4-a_5)(\alpha_1-\alpha_2)(\gamma_2-\beta_1)(\gamma_2-\beta_2)(\gamma_2-\gamma_1)\\
    &= \frac{a_2+a_3-a_4-a_5}{D_2-D_3} \cdot \frac{4(D_2-D_3)\sqrt{D_3D_1}}{a_2+a_3-a_4-a_5} = 4\sqrt{D_3D_1}.
\end{align*}
Then $\kappa$ is a square in $\mathbb{F}_{p^2}\hspace{-0.3mm}$ by using Lemma \ref{fourth}, and hence this proposition is true (we note that we need only prove that for {\it one} Rosenhain form of $C'$ by Lemma \ref{alltoone}).
\end{proof}

\begin{proof}[\underline{Proof of Main Theorem B}]
By using Proposition \ref{C2max}, this assertion holds for $C$ such that the Jacobian variety of $C$ is $(2,2)$-isogenous to the product of two elliptic curves. Moreover, by using Proposition \ref{isomax}, this assertion holds also for $C'$ such that $C$ and $C'$ are Richelot isogenous. By doing this repeatedly, we complete the proof for all genus-2 curves (recall from Algorithm \ref{Algo} that this procedure ends in finite times).
\end{proof}

\section{Applications of Main Theorems}\label{Application}
In this section, we give some results obtained by applying Main Theorems. In Subsection 4.1, we give another proof that there does not exist a superspecial genus-2 curve of characteristic $p=3$. In Subsections 4.2 and 4.3, we will show that similar results as Main Theorem B hold for superspecial genus-3 and genus-4 hyperelliptic curves whose automorphism groups contain $\mathbb{Z}/2\mathbb{Z} \times \mathbb{Z}/2\mathbb{Z}$.

\subsection{Another proof of non-existence of superspecial genus-2 curve for $p=3$}\label{another}
Ibukiyama-Katsura-Oort \cite{IKO} showed that there are no superspecial genus-2 curves in characteristic $p=3$, by computing the class numbers of quaternion hermitian forms. In the following, we give another proof of this. The next corollary holds for general $p \geq 3$.
\begin{Cor}\label{S}
Let $S \subset \mathbb{F}_{p^2}\hspace{-0.3mm}$ be the set of all elements $s \neq 0,1$ such that both $s$ and $1-s$ are squares in $\mathbb{F}_{p^2}$. If the genus-2 curve $C: y^2 = x(x-1)(x-\lambda)(x-\mu)(x-\nu)$ is superspecial, then $\lambda,\mu$ and $\nu$ belong to $S$.
\end{Cor}
\begin{proof}
This is a direct consequence of Main Theorem\hspace{1mm}A(1).
\end{proof}
\begin{Thm}[{\cite[Thoerem\ 3.3]{IKO}}]
No superspecial genus-2 curves exist in characteristic $3$.
\end{Thm}
\begin{proof}
Assume that $C: y^2 = x(x-1)(x-\lambda)(x-\mu)(x-\nu)$ is superspecial. For $p=3$, we see that the set $S$ in Corollary \ref{S} is given as $S=\{2\}$ by a simple computation. This means that $\lambda = \mu = \nu = 2$, which leads to a contradiction since $C$ has a singular point.
\end{proof}

\subsection{Application to genus-3 hyperelliptic curves}\label{genus3}
Moriya-Kudo studied the superspeciality of genus-3 hyperelliptic curves $D$ such that ${\rm Aut}(D) \hspace{-0.3mm}\supset \mathbb{Z}/2\mathbb{Z} \times \mathbb{Z}/2\mathbb{Z}$ in \cite{MK}. They showed that such $D$ can be written as
\[
    D: y^2 = (x^2-1)(x^2-a)(x^2-b)(x^2-c)
\]
for $a,b,c \in K$, and computed the number of isomorphism classes of superspecial $D$ for small primes $p \leq 200$. In the following, we show that if $D$ is superspecial, then $a,b,c$ belongs to $\mathbb{F}_{p^2}\hspace{-0.3mm}$ and moreover $D$ is maximal or minimal over $\mathbb{F}_{p^2}$.
\begin{Thm}
Assume that a genus-3 hyperelliptic curve
\[
    D: y^2 = (x^2-1)(x^2-a)(x^2-b)(x^2-c)
\]
is superspecial. Then, we have the following statements:
\begin{enumerate}
\item Each $a,b,c$ is a square in $\mathbb{F}_{p^2}$.\vspace{-1mm}
\item If $p \equiv 3 \pmod{4}$, then the curve $D$ is maximal over $\mathbb{F}_{p^2}$. Otherwise, the curve $D$ is minimal over $\mathbb{F}_{p^2}$.
\end{enumerate}
\end{Thm}
\begin{proof}
As shown in \cite[Section\,2]{MK}, the curve $D$ is birational to the fiber product $E \times_{\mathbb{P}^1\!} C$ where
\begin{alignat*}{2}
    E&: Y^2 = (X-1)(X-a)(X-b)(X-c) & {\rm \ \,with\ } & X = x^2,\,Y = y,\\
    C&: Y^2 = X(X-1)(X-a)(X-b)(X-c) & {\rm \,with\ } & X = x^2,\,Y = xy.
\end{alignat*}
By the assumption that $D$ is superspecial, then we have that $E$ is supersingular and $C$ is also superspecial. We consider the change of variables\vspace{-1mm}
\begin{equation}\label{trans}
    X \rightarrow \frac{X-1}{X-a} \cdot \frac{b-a}{b-1} =: u, \quad Y \mapsto \frac{Y}{(a-1)(a-b)(a-c)}\biggl(\frac{a-1}{X-a} \cdot \frac{b-a}{b-1}\biggr)^{\!\!2} =: v.\vspace{-0.5mm}
\end{equation}
This transformed the curve $E$ into the form\vspace{-0.5mm}
\[
    \kappa v^2 = u(u-1)(u-\lambda), \quad \lambda := \frac{(b-a)(c-1)}{(b-1)(c-a)}\vspace{-0.5mm}
\]
with $\kappa = -(1-b)(c-a)$.\par\vspace{1mm}
\hspace{5.5mm}(1) Using Main Theorem A(1), all the 9 values
\[
    a,\,b,\,c,\,1-a,\,1-b,\,1-c,\,a-b,\,b-c,\,c-a
\]
are squares in $\mathbb{F}_{p^2}\hspace{-0.3mm}$. Hence, we obtain the first assertion of this theorem.\par\vspace{1mm}
\hspace{5.5mm}(2) Using Auer-Top's result \cite[Proposition\hspace{0.8mm}2.2]{AT}, a supersingular elliptic curve $v^2 = u(u-1)(u-\lambda)$ is maximal (resp. minimal) over $\mathbb{F}_{p^2}\hspace{-0.5mm}$ when $p \equiv 3$ (resp. $p \equiv 1$), and so is the elliptic curve $\kappa v^2 = u(u-1)(u-\lambda)$ since $\kappa = -(1-b)(c-a)$ is a square in $\mathbb{F}_{p^2}$. We see that $E$ is maximal (resp. minimal) over $\mathbb{F}_{p^2}\hspace{-0.5mm}$ when $p \equiv 3$ (resp. $p \equiv 1$) since the transformation of (\ref{trans}) is defined over $\mathbb{F}_{p^2}$. Moreover $C$ is maximal (resp. minimal) over $\mathbb{F}_{p^2}\hspace{-0.5mm}$ when $p \equiv 3$ (resp. $p \equiv 1$) by using Main Theorem B. As the birational map $D \rightarrow E \times_{\mathbb{P}^1\!} C$ is defined over $\mathbb{F}_{p^2}$, hence this theorem is true.
\end{proof}

\subsection{Application to genus-4 hyperelliptic curves}
Similarly to the genus-3 case, Ohashi-Kudo-Harashita \cite{OKH} studied the superspeciality of genus-4 hyperellitpic curves $D'$ satisfying ${\rm Aut}(D') \hspace{-0.3mm}\supset \mathbb{Z}/2\mathbb{Z} \times \mathbb{Z}/2\mathbb{Z}$. They showed that such $D'$ can be written as
\[
    D': y^2 = (x^2-1)(x^2-a)(x^2-b)(x^2-c)(x^2-d)
\]
for $a,b,c,d \in K$, and computed the number of isomorphism classes of superspecial $D'$ for all primes $p \leq 200$. They also expected \cite[Remark\,3]{OKH} that superspecial $D'$ are all maximal or minimal over $\mathbb{F}_{p^2}$. In the following, we prove their conjecture (Theorem \ref{genus4}). In addition, we give a simple criterion in terms of $a,b,c,d$ that tells whether $D'$ is maximal or minimal over $\mathbb{F}_{p^2}\!$ (Corollary \ref{criterion}).
\begin{Thm}\label{genus4}
Assume that a genus-4 hyperelliptic curve
\[
    D': y^2 = (x^2-1)(x^2-a)(x^2-b)(x^2-c)(x^2-d)
\]
is superspecial. Then, we have the following statements:
\begin{enumerate}
    \item Each $a,b,c,d$ is a square in $\mathbb{F}_{p^2}$.\vspace{-1mm}
    \item The curve $D'$ is maximal or minimal over $\mathbb{F}_{p^2}$. 
\end{enumerate}
\end{Thm}
\begin{proof}
As shown in \cite[Section\,3]{OKH}, the curve $D'$ is birational to the fiber product $C_1 \times_{\mathbb{P}^1\!} C_2$ where
\begin{alignat*}{2}
    C_1&: Y^2 = (X-1)(X-a)(X-b)(X-c)(X-d),  & {\rm \ \,with\ } & X = x^2, \  Y = y,\\
    C_2&: Y^2 = X(X-1)(X-a)(X-b)(X-c)(X-d)  & {\rm \ \,with\ } & X = x^2, \  Y = xy.
\end{alignat*}
By the assumption that $D'$ is superspecial, then we obtain that two curves $C_1$ and $C_2$ are also superspecial. We consider the change of variables\vspace{-1mm}
\begin{equation}\label{trans1}
    X \rightarrow \frac{X-1}{a-1} =:u, \quad Y \mapsto \frac{Y}{(a-1)^3} =: v.\vspace{-1mm}
\end{equation}
This transformed the curve $C_1$ into the form\vspace{-1mm}
\[
    \kappa v^2 = u(u-1)(u-\lambda)(u-\mu)(u-\nu), \quad \lambda := \frac{b-1}{a-1}, \quad \mu := \frac{c-1}{a-1}, \quad \nu := \frac{d-1}{a-1},
\]
with $\kappa = -(1-a)$. On the other hand, the change of variables\vspace{-1.5mm}
\begin{equation}\label{trans2}
    X \rightarrow \frac{X-1}{X} \cdot \frac{a}{a-1} =: u, \quad Y \mapsto \frac{Y}{abcd}\biggl(\frac{1}{X} \cdot \frac{a}{a-1}\biggr)^{\!\!3} =: v\vspace{-0.5mm}
\end{equation}
transformed the curve $C_2$ into the form\vspace{-1mm}
\[
    \kappa' v^2 = u(u-1)(u-\lambda')(u-\mu')(u-\nu'), \quad \lambda' := \frac{a(b-1)}{b(a-1)}, \quad \mu' := \frac{a(c-1)}{c(a-1)}, \quad \nu' := \frac{a(d-1)}{d(a-1)},\vspace{-0.5mm}
\]
with $\kappa' = -(1-a)bcd$.\par\vspace{1mm}
\hspace{5.5mm}(1) Using Main Theorem A(1), all values $\lambda,\mu,\nu,1-\lambda,1-\mu,1-\nu$ and $\lambda',\mu',\nu',1-\lambda',1-\mu',1-\nu'$ are squares in $\mathbb{F}_{p^2}$. Here, one can compute\vspace{-1mm}
\[
    a = \frac{\lambda'(1-\lambda)}{\lambda(1-\lambda')}, \quad b = \frac{1-\lambda}{1-\lambda'}, \quad c = \frac{\mu}{\mu'} \cdot \frac{\lambda'(1-\lambda)}{\lambda(1-\lambda')}, \quad d = \frac{\nu}{\nu'} \cdot \frac{\lambda'(1-\lambda)}{\lambda(1-\lambda')},
\]
and hence $a,b,c$ and $d$ are all squares in $\mathbb{F}_{p^2}$.\par\vspace{1mm}
\hspace{5.5mm}(2) We divide into two cases depending on whether $1-a$ is a square in $\mathbb{F}_{p^2}\hspace{-0.3mm}$ or not. Recall that two curves $v^2 = u(u-1)(u-\lambda)(u-\mu)(u-\nu)$ and $v^2 = u(u-1)(u-\lambda')(u-\mu')(u-\nu')$ by Main Theorem B.
\begin{itemize}
    \item If $1-a$ is a square in $\mathbb{F}_{p^2}$, then two values $\kappa = 1-a$ and $\kappa' = -(1-a)bcd$ are squares in $\mathbb{F}_{p^2}$. Since two transformations (\ref{trans1}) and (\ref{trans2}) are defined in $\mathbb{F}_{p^2}$, and thus $C_1$ and $C_2$ are maximal (resp. minimal) over $\mathbb{F}_{p^2}\hspace{-0.3mm}$ if and only if $p \equiv 3$ (resp. $p \equiv 1$). As the birational map $D' \rightarrow C_1 \times_{\mathbb{P}^1\!} C_2$ is defined over $\mathbb{F}_{p^2}$, and hence $D'$ is also maximal (resp. minimal) over $\mathbb{F}_{p^2}$ when $p \equiv 3$ (resp. $p \equiv 1$).
    \item If $1-a$ is not a square in $\mathbb{F}_{p^2}$, then $\kappa = 1-a$ and $\kappa' = -(1-a)bcd$ are also not squares in $\mathbb{F}_{p^2}$. Since two transformations (\ref{trans1}) and (\ref{trans2}) are defined in $\mathbb{F}_{p^2}$, and thus $C_1$ and $C_2$ are minimal (resp. maximal) over $\mathbb{F}_{p^2}\hspace{-0.5mm}$ if and only if $p \equiv 3$ (resp. $p \equiv 1$). As the birational map $D' \rightarrow C_1 \times_{\mathbb{P}^1\!} C_2$ is defined over $\mathbb{F}_{p^2}$, and hence $D'$ is also minimal (resp. maximal) over $\mathbb{F}_{p^2}$ when $p \equiv 3$ (resp. $p \equiv 1$).\vspace{-0.5mm}
\end{itemize}
In any case, this theorem is true.
\end{proof}
\begin{Cor}\label{criterion}
Suppose that $D'$ is superspecial, then the following are true:
\begin{itemize}
\item If $p \equiv 3 \pmod{4}$, then $D'$ is maximal if and only if a/all $1-a,1-b,1-c,1-d$ is a square in $\mathbb{F}_{p^2}$.\vspace{-1mm}
\item If $p \equiv 1 \pmod{4}$, then $D'$ is maximal if and only if a/all $1-a,1-b,1-c,1-d$ is not a square in $\mathbb{F}_{p^2}$.
\end{itemize}
\end{Cor}
\begin{proof}
With notations in the proof of Theorem \ref{genus4}, all values\vspace{-1.5mm}
\[
    \lambda = \frac{b-1}{a-1}, \quad \mu = \frac{c-1}{a-1}, \quad \nu = \frac{d-1}{a-1}\vspace{-1mm}
\]
are squares in $\mathbb{F}_{p^2}$. This implies all $1-a,1-b,1-c,1-d$ are squares or none of these is a square. Hence, this corollary directly follows from the proof of Theorem \ref{genus4}\,(2).
\end{proof}

\bigskip
\begin{flushright}
Ryo Ohashi \\
Graduate School of Information Science and Technology, \\
The University of Tokyo, \\
7-3-1 Hongo, Bunkyo-ku, Tokyo, 113-8656, \\
Japan. \\
E-mail: \texttt{ryo-ohashi@g.ecc.u-tokyo.ac.jp}
\end{flushright}
\end{document}